\documentclass{article}
\usepackage{newton3p}
\usepackage{graphicx}
\usepackage{hyperref}
\usepackage{algorithmic}

\usepackage{amsmath,amssymb,amsthm,enumerate,mathrsfs}
\usepackage{cite}
\newtheorem{theorem}{Theorem}[section]

\newtheorem{lemma}[theorem]{Lemma}

\theoremstyle{definition}

\newtheorem{algo}[theorem]{Algorithm}

\newtheorem{algorithm}{Algorithm}[section]
\usepackage{algorithmic}

\numberwithin{equation}{section}

\begin{document}
\makeatletter

\begin{center}
\large{\bf Newton's Method in Three Precisions}
\end{center}\vspace{5mm}
\begin{center}
\textsc{
C. T. Kelley\\
North Carolina State University\\
Department of Mathematics\\
Box 8205, Raleigh, NC 27695-8205, USA\\
Tim\_Kelley@ncsu.edu
}\end{center}

\vspace{2mm}

\footnotesize{
\noindent\begin{minipage}{14cm}
{\bf Abstract:}
We describe a three precision variant of Newton's method for nonlinear 
equations. We evaluate the nonlinear residual in double precision, store
the Jacobian matrix in single precision, and solve the equation for the 
Newton step with iterative refinement with a factorization in half precision.
We analyze the method as an inexact Newton method. This analysis shows that,
except for very poorly conditioned Jacobians,
the number of nonlinear iterations needed is the same that one would get if
one stored and factored the Jacobian in double precision. In many
ill-conditioned cases one can use the low precision factorization as a
preconditioner for a GMRES iteration. That approach can recover fast
convergence of the nonlinear iteration.
We present an example to illustrate the results.
\end{minipage}
 \\[5mm]
\noindent\begin{minipage}{14cm}
{\bf Dedication:}
To Masao Fukushima on his 75th birthday, with thanks for his many contributions to optimization.
\end{minipage}
}

\noindent{\bf Keywords:} {Newton's Method, Iterative Refinement, 
Mixed-Precision Algorithms}

\noindent{\bf Mathematics Subject Classification:
65H10, 
65F05, 
65F10, 
45G10, 
}

\hbox to14cm{\hrulefill}\par


\section{Introduction}

This paper is about using low precision arithmetic in the computation of
a Newton step. In many cases the cost of the factorization of the Jacobian is
$O(N^3)$, where $N$ is the number of unknowns, and this dominates the cost of
the computation. Therefore, computation of the factorization in lower precision
offers significant savings in cost.

The new algorithms in this paper
are based on prior results from \cite{ctk:fajulia, ctk:sirev20}.
In that older work we implemented Newton's method for the nonlinear equation
in two precisions. We evaluated the nonlinear residual in double precision 
and stored and factored the Jacobian in a lower precision 
(either single or half). 
The resulting algorithm is
a nonlinear version of the
classic iterative refinement method from linear algebra
\cite{higham,wilkinson63}.

In \cite{ctk:sirev20} we observed that using
half precision, while offering the potential of faster computation, 
degraded
the performance of the nonlinear iteration for ill-conditioned problems.
In this paper we suggest a way to address that issue by storing copies
of the Jacobian in both single and half precision, factoring the Jacobian
in half precision, and solving the equation for the Newton step
in single precision using the half precision factorization within 
iterative refinement (IR). We can make this process more robust by using the
GMRES-IR method of \cite{CarsonHigham1,CarsonHigham}, where the low
precision factorization is used as a preconditioner for a GMRES 
iteration for the high precision problem. 

We view this as an inexact Newton method \cite{demboes,ctk:roots} 
and the nonlinear convergence analysis will depend on the performance
of the underlying linear iterative method, which will be IR or GMRES-IR.

\subsection{Notation and basic results}
\label{subsec:notation}
We consider a nonlinear equation
\begeq
\label{eq:nleq}
\mf (\vx) = 0
\endeq
for $\vx \in \Omega \subset R^N$.
We will call $\mf$ the residual in this paper. We denote the Jacobian matrix
of $\mf$ by $\mf'$.

We will assume that the standard assumptions
\cite{dens,ctk:roots,ctk:fajulia} hold.

{\bf Standard Assumptions}
\begin{enumerate}
\item Equation~{\rm\ref{eq:nleq}} has a solution $\vx^* \in \Omega$.
\item $\mf': \Omega \to R^{N \times N}$ is Lipschitz continuous
near $\vx^*$ with Lipschitz constant $\gamma$.
\item $\mf'(\vx^*)$ is nonsingular.
\end{enumerate}

We will assume that we are near enough
to $\vx^*$ so that the Newton iteration
\begeq
\label{eq:newtonit}
\vx_{+} = \vx_c - \mf'(\vx_c)^{-1} \mf(\vx_c)
\endeq
will converge quadratically to the solution.  In \eqnok{newtonit}
$\mf'$ is the Jacobian matrix
and, as is standard, $\vx_c$ denotes the current point and
$\vx_+$ denotes the Newton iteration from $\vx_c$.

One explicit way \cite{ctk:roots}
to express this is to assume that
\begeq
\label{eq:close}
\vx_c \in \calb \equiv \{ \vx \, | \, \| \vx - \vx^* \| \le  \rho \}
\endeq
where
\[
\rho \le \frac{1}{2 \gamma \| \mf'(\vx^*)^{-1} \|}
\]
and is small enough so that $\calb \subset \Omega$. In that case
$\vx_+ \in \calb$ and 
\[
\| \ve_+ \| \le \gamma \| \mf'(\vx^*)^{-1} \| \| \ve_c \|^2
\le \| \ve_c \|/2 \le \rho/2.
\]
And so the iteration converges and remains in $\Omega$.
Here $\ve = \vx - \vx^*$ denotes the error.

In practice, however, the Newton iteration is computed in
floating point arithmetic and the floating point errors must
be considered.
To account for this (see \cite{ctk:roots} for the details)
we let $\Delta$ denote the error in the Jacobian and $\epsilon$ denote
the error in the residual. With this in mind the iteration is
\begeq
\label{eq:realit}
\vx_+ = \vx_c - (\mf'(\vx_c) + \Delta(\vx_c))^{-1} (\mf(\vx_c)
+ \epsilon(\vx_c))
\endeq

In this paper we will assume that the errors can be bounded independently of
$\vx$, so there are $\epsilon_F$ and $\epsilon_J$ such that
\[
\| \epsilon(\vx) \| \le \epsilon_F \mbox{ and }
\| \Delta(\vx) \| \le \epsilon_J
\]
for all $\vx \in \calb$. One can think
of $\epsilon_F$ as floating point roundoff. The interesting part is
$\epsilon_J$, the error in the Jacobian.

With this in mind, the error estimate from \cite{ctk:roots} becomes
\begeq
\label{eq:errest}
\| \ve_+ \| = O(\| \ve_c \|^2 + \epsilon_J \| \ve_c \|
+ \epsilon_F ),
\endeq
where $\ve = \vx - \vx^*$ denotes the error,
Clearly, if the errors vanish, one obtains the standard quadratic
convergence theory. However, if $\epsilon_F > 0$ then one can expect
the residual norms to stagnate once
\[
\| \mf(\vx) \| = O(\epsilon_F)
\]
which is what one observes in practice. The estimate \eqnok{errest}
is not a local convergence result. Results of this type are called
{\em local improvement} \cite{denwal3,ctk:roots}.

We can also see that if $\epsilon_J = O(\sqrt{\epsilon_F})$, as it will
be \cite{ctk:roots} if one uses a finite-difference approximation to the
Jacobian with difference increment $\sqrt{\epsilon_F}$ or (assuming one
computes $\mf$ in double precision) stores and factors
the Jacobian in single precision,
then \eqnok{errest} becomes
\begeq
\label{eq:errest2}
\| \ve_+ \| 
= O(\| \ve_c \|^2 + \sqrt{\epsilon_F} \| \ve_c \| + \epsilon_F )
= O(\| \ve_c \|^2 + \epsilon_F ).
\endeq
Equation \eqnok{errest2} says that the iteration with a sufficiently 
accurate approximate Jacobian is indistinguishable from Newton's
method.

With these errors in mind, we can formulate the locally
convergent (\ie with no line search) form of Newton's method of interest
in this paper. 
In Algorithm~\ref{alg:newton} $\tau_a$ and $\tau_r$ are relative and
absolute error tolerances. $\vx$ is the initial iterate on input and
the algorithm overwrites $\vx$ as the iteration progresses.

\begin{algo}
\label{alg:newton}
{$\mbox{\bf newton}(\mf, \vx, \tau_a, \tau_r)$} 
\begin{algorithmic}
\STATE Evaluate ${\tilde \mf} = \mf(\vx) + \epsilon(\vx)$; 
\STATE $\tau \leftarrow \tau_r \| {\tilde \mf} \| + \tau_a$.
\WHILE{$\| {\tilde \mf} \| > \tau$}
\STATE Solve $(\mf'(\vx) + \Delta(\vx)) \vs = - {\tilde \mf}$
\STATE $\vx \leftarrow \vx + \vs$
\STATE Evaluate ${\tilde \mf} = \mf(\vx) + \epsilon(\vx)$; 
\ENDWHILE
\end{algorithmic}
\end{algo}

There are two sources of error in the Jacobian that contribute
to $\Delta$. One is the error in approximating the Jacobian itself
and the other is the error in the solver. We will use an analytic
Jacobian in this paper and store that Jacobian in single precision.
Hence the relative error in the Jacobian is
floating point roundoff in single precision.
We used Gaussian elimination to solve for the Newton
step in \cite{ctk:sirev20} so the solver error was the backward
error in the $LU$ factorization. We will discuss this more in
\S~\ref{subsec:sirev}.

\subsection{IEEE Arithmetic}
\label{subsec:ieee}

We remind the reader of some details of IEEE floating point arithmetic
\cite{overtonbook,IEEEnew}. The standard precisions in most software
environments are single and double precision. Half precision was originally
proposed as a storage format \cite{ieee} and is not implemented in
hardware on many platforms.
We will describe the details
of these three precisions in terms of the amount of storage a floating
point number requires (the width) and the unit roundoff $u$. As is standard
\cite{higham} we define
$u$ in terms of the floating point error in rounding the result of
any binary operation $\circ = \pm, \times, \div$ applied to two floating
point numbers $x$ and $y$
\[
fl( x \circ y ) = (x \circ y) ( 1 + \delta ), \  | \delta | \le u.
\]
Here $fl$ is the rounding map which takes a real $z$ in the range
of the floating point number system to the nearest floating point number.
If $z$ is not in the range of the floating point number system, then
attempting to compute $fl(z)$ will generate an exception. The range
will be important in this paper because one must pay particular
attention to that
when computing a Newton step with a half-precision Jacobian. We define
the range of the floating point number system as
\[
{\cal R} = \{ z \ | \ \sigma_L \le | z | \le \sigma_H \}
\]
where $\sigma_L$ is the smallest positive floating point number and
$\sigma_H$ is the largest positive floating point number.

We can now summarize the properties of the three precisions in this paper.
We took the data in Table~\ref{tab:precision} from a similar table in
\cite{NickCSE}.

\begin{table}[h!]
\caption{IEEE precisions}
\label{tab:precision}
\begin{center}
\begin{tabular}{|l|l|l|l|l|l}
\hline
Precision & width (bits) & $u$ & $\sigma_L$ & $\sigma_H$ \\
\hline
Half &  $16$ & $\approx 5 \times 10^{-4}$ & $10^{-5}$ & $10^5$\\
\hline
Single &  $32$ & $\approx 6 \times 10^{-8}$ & $10^{-38}$ & $10^{38} $ \\
\hline
Double &  $64$ & $\approx 10^{-16}$ & $10^{-308}$ & $10^{308} $ \\
\hline
\end{tabular}
\end{center}
\end{table}
When we discuss multiprecision computations we will let 
$u_d, u_s, u_h$ be  unit roundoff in double, single, or half precision.

Double and single precisions have been
supported in hardware for decades. Recently new computer architectures such
as the Apple M1 and M2 chips have been offering hardware support for half
precision. However, tools such as LAPACK and the BLAS 
\cite{lapack} do not support half precision
yet and run far more slowly in half precision
than they do in double precision. There is
active research on extending
the BLAS and LAPACK to use half precision \cite{newblas}.
The algorithm we propose in this paper will exploit half
precision well once the tools catch up.

\subsection{Newton's method in two precisions}
\label{subsec:sirev}

With the background from the previous sections in hand, we can
now describe the findings from \cite{ctk:sirev20} and then motivate
the three precision algorithms.

As we said above, the equation for the Newton step
\[
\mf'(\vx_c) \vs = - \mf(\vx_c)
\]
can only be approximated. The first step in such an approximation
is to replace $\mf'(\vx_c)$ with an approximation $\mj$. For
example $\mj$ could be a floating point evaluation of the Jacobian,
perhaps in a lower precision that the one used to evaluate $\mf$,
a finite difference approximation, or a physics-based approximation
that neglects part of the Jacobian. In any case, one can analyze
the error in $\mj$ directly. 

In this work we solve the approximation
\[
\mj \vs = - \mf(\vx_c)
\]
with Gaussian elimination \cite{higham}, \ie an $LU$ factorization.
We compute an upper triangular matrix $\mU$ and a lower triangular
matrix $\ml$ that, in exact arithmetic, factors $\mj = \ml \mU$,
so the equation for the step can be solved by two triangular solves.

However, there are errors in factorization and one really computes
approximations ${\hat \ml}$ and  ${\hat \mU}$. We define
${\hat \mj} = {\hat \ml} {\hat \mU}$, so the approximate factorization
is the exact factorization for a different (hopefully nearby) problem.
Hence, the equation we actually solve for the Newton step is
\[
{\hat \mj} \vs = - \mf(\vx_c).
\]
There is a subtle point in the equation for the Newton step. The
matrix ${\hat \mj}$ may be in a different precision than $\mf$ and 
$\vs$. One must take some care with this and we return to this point
in \S~\ref{sec:IR} and \ref{sec:ipxf}.

The backward error is
\[
\delta \mj = {\hat \mj} - \mj.
\]
So the error in the Jacobian ($\epsilon_J$ in \eqnok{errest}) has
has two parts, the error in 
$\mj$ and the backward error in the factorization.

We will assume for this paper that $\mf$ is computed in double precision,
so $\epsilon_F$ is $O(u_d)$. The error one makes is storing the
Jacobian in reduced precision is 
\[
\| \mj - \mf'(\vx_c) \| \le u_J,
\]
where $u_J = O(u_s)$ or $O(u_h)$. We will make the
contribution from the backward error explicit and
reformulate \eqnok{errest} as
\[
\| \ve_+ \| =
O( \| \ve_c \|^2 + ( u_J  
+ \| \delta \mj \|) \| \ve_c \| + \epsilon_F ).
\]

The results in \cite{ctk:sirev20} show that if the Jacobian
is stored and factored in single precision and the size $N$ of
the problem is not too large,
then there is
no difference in the iteration statistics from storing
and factoring the Jacobian in double precision. So both the approximation
error and the backward error in the solver are $O(u_d)$.
However,
if the Jacobian is stored and factored in half precision, there
are differences caused by the poor accuracy of half precision,
and the nonlinear iteration can converge slowly
or even fail to converge.

The new algorithms in this paper use a half precision factorization
as part of an iterative method to compute a Newton step with a
single precision Jacobian. This approach, as we explain in 
\S~\ref{sec:3p}, requires some care.

\section{Three precision algorithms}
\label{sec:3p}

The three precision algorithms compute $\mf$ in double precision and
store the Jacobian $\mf'$ in a single precision matrix $\mj$. This means that
\begeq
\label{eq:jacerr}
\| \mf'(\vx_c) - \mj \| \le u_s \| \mf'(\vx_c) \|.
\endeq
Hence, using the terminology of \S~\ref{subsec:notation}
\begeq
\label{eq:errlist}
\epsilon_F = O(u_d) \mbox{ and } \epsilon_J = O(u_s).
\endeq

Our notation for interprecision transfers is to let 
$I_a^b$ be the transfer from precision $u_a$ to
$u_b$. If $u_a > u_b$, this promotion changes nothing
\[
I_a^b (x) = x
\]
if $x$ is in precision $u_a$. If $u_a < u_b$, then the interprecision
transfer rounds down, so
\[
\| I_a^b(x) - x \| \le u_b \| x \|.
\]
We will use these properties of interprecision transfer throughout
the remainder of the paper. We point out that when one rounds a
matrix or vector down to a lower precision, one must allocate memory
for the low precision object and that there is a cost to this.

We then round 
$\mj$ to half precision to obtain
\[
\mj_h = I_s^h (\mj)
\]
and factor $\mj_h$ in half precision to obtain
${\hat \ml} {\hat \mU}$.
We use the
half precision factorization as part of an iterative method to solve
\[
\mj \vs = - \mf(\vx_c).
\]
We terminate that iteration when 
\begeq
\label{eq:inexactJ}
\| \mj \vs + \mf(\vx_c) \| \le \eta_J \| \mf(\vx_c) \|.
\endeq

We summarize the three precision algorithm.
\begin{algo}
\label{alg:3p}
{$\mbox{\bf newton3p}(\mf, \vx, \tau_a, \tau_r, \eta_J)$}
\begin{algorithmic}
\STATE Evaluate ${\tilde \mf} = \mf(\vx) + \epsilon(\vx)$;  
\STATE $\tau \leftarrow \tau_r \| {\tilde \mf} \| + \tau_a$.
\WHILE{$\| {\tilde \mf} \| > \tau$}
\STATE Compute and store $F'(\vx)$ in single precision as $\mj$.
\STATE Store $\mj_h = I_s^h(\mj)$.
\STATE Find $\vs$ such that $\| \mj_h \vs + \mf(\vx) \| \le 
\eta_J \| \mf(\vx) \|$. 
\STATE $\vx \leftarrow \vx + \vs$
\STATE Evaluate ${\tilde \mf} = \mf(\vx) + \epsilon(\vx)$;
\ENDWHILE
\end{algorithmic}
\end{algo}
In Algorithm~\ref{alg:3p} we want to choose $\eta_J < 1$ small enough
so that the nonlinear iteration statistics are the same as those from
Newton's method itself.

Algorithm~\ref{alg:3p} looks like an inexact Newton iteration, but 
differs in that the condition on the step is \eqnok{inexactJ}
rather than the classical inexact Newton condition
\begeq
\label{eq:inexact}
\| \mf'(\vx_c) \vs + \mf(\vx_c) \| \le \eta \| \mf(\vx_c) \|.
\endeq
If we had \eqnok{inexact}, then we would get a local improvement 
estimate \cite{demboes,ctk:roots}
\begeq
\label{eq:inexloc}
\| \ve_+ \| = O(\| \ve_c \|^2 + \eta \|\ve_c\| + \epsilon_F).
\endeq
This will imply q-linear convergence of the nonlinear
iteration if $\eta$ is sufficiently small and the function
evaluation is exact ($\epsilon_F = 0$).

If we are able to show that we can chose $\eta_J$ so that \eqnok{inexact}
holds with $\eta = O(u_s)$, then, similar to the two precision case with
$\mj$ stored and factored in single precision, \eqnok{inexloc} will imply
\eqnok{errest2} and the nonlinear iteration
statistics will be the same as Newton's method with the Jacobian stored and
factored in double precision.

The use of an iterative method for the linear equation for the Newton
step means that the backward error in the 
factorization plays no role in the analysis of the nonlinear iteration.
However, that backward error
does affect the convergence of the linear iteration. We will describe
our two choices for the linear iteration in \S~\ref{sec:IR} but will
discuss the local improvement result for the
nonlinear iteration first.

\subsection{Local improvement of the nonlinear iteration}
\label{subsec:inexact}

We begin by showing that $\mj$ is nonsingular and estimating
$\| \mj^{-1} \|$. In the analysis we use the standard notation
\[
\kappa(\ma) = \| \ma \| \| \ma^{-1} \|
\]
for the condition number of a matrix $\ma$.

\begin{lemma}
\label{lem:mjok}

Assume that the standard assumptions \eqnok{close} hold and that 
\begeq
\label{eq:condok}
4 u_s \kappa(\mf'(\vx^*) ) < 1.
\endeq

Then $\mj$ is nonsingular and
\begeq
\label{eq:mjok}
\| \mj^{-1} \| \le
\frac{2 \| \mf'(\vx^*)^{-1} \|}{1 - 4 u_s \kappa(\mf'(\vx^*) )}.
\endeq
\end{lemma}

\begin{proof}

The standard assumptions and \eqnok{close} imply that
$\mf'(\vx_c)$ is nonsingular and
(see Lemma 4.3.1 from \cite{ctk:roots})
\begeq
\label{eq:twotimes}
\| \mf'(\vx_c) \| \le 2 \| \mf'(\vx^*) \|  \mbox{ and }
\| \mf'(\vx_c)^{-1} \| \le 2 \| \mf'(\vx^*)^{-1} \|.
\endeq
Hence, using \eqnok{twotimes},
\[
\| I - \mf'(\vx_c)^{-1} \mj \|
\le \| \mf'(\vx_c)^{-1} \| \| \mf'(\vx_c) - \mj \|
\le u_s \| \mf'(\vx_c)^{-1} \| \| \mf'(\vx_c) \|
\le 4 u_s \kappa(\mf'(\vx^*) ) < 1.
\]

So $\mf'(\vx_c)^{-1}$ is an approximate inverse of $\mj$. Therefore $\mj$
is nonsingular and
\[
\| \mj^{-1} \| \le 
\frac{\| \mf'(\vx_c)^{-1} \|}{1 - 4 u_s \kappa(\mf'(\vx^*) )}
\le
\frac{2 \| \mf'(\vx^*)^{-1} \|}{1 - 4 u_s \kappa(\mf'(\vx^*) )},
\]
proving the lemma.

\end{proof}

Assume the linear iterative method converges, which is not guaranteed, 
and that we terminate the linear iteration when 
\eqnok{inexactJ} holds. To prove the local improvement estimate
\eqnok{inexloc} we must connect
\eqnok{inexactJ} to the classic inexact Newton condition
\eqnok{inexact} for some $\eta < 1$. That will then imply
the estimate \eqnok{inexloc}.

We express the convergence estimates in terms of 
\begeq
\label{eq:pstar}
P^* = \frac{4 \| \kappa(\mf'(\vx^*)) \|}%
{1 - 4 u_s \kappa(\mf'(\vx^*) )}.
\endeq

\begin{lemma}
\label{lem:etaok}
Assume that the assumptions of Lemma~\ref{lem:mjok} and \eqnok{inexactJ} 
hold, that $u_s P^* < 1/2$, and that
\[
\eta_J < 1 - 2 u_s P^*.
\]
Then \eqnok{inexact} holds with
\begeq
\label{eq:etaok}
\eta \le \eta_J + (1 + \eta_J) u_s P^* < 1.
\endeq
\end{lemma}

\begin{proof}

Equation \eqnok{inexactJ} implies that
\[
\| \mj^{-1} \|^{-1} \| \vs \| \le \| \mj \vs \|
\le (1 + \eta_J) \| \mf(\vx_c) \|
\]
and hence, using Lemma~\ref{lem:mjok}
\begeq
\label{eq:stepest}
\| \vs \| \le \| \mj^{-1} \| (1 + \eta_J) \| \mf(\vx_c) \| 
\le \frac{2 \| \mf'(\vx^*)^{-1} \|}{1 - 4 u_s \kappa(\mf'(\vx^*) )}
(1 + \eta_J)  \| \mf(\vx_c) \|.
\endeq

We use \eqnok{inexactJ} again to obtain
\begeq
\label{eq:inexactok}
\begin{array}{ll}
\| \mf'(\vx_c) \vs + \mf(\vx_c) \| 
& \le 
\| \mj \vs + \mf(\vx_c) \| + 
\| (\mf'(\vx_c) - \mj) \vs \| \\
\\
& \le \eta_J \| \mf(\vx_c) \| + u_s \| \mf'(\vx_c) \| \| \vs \| \\
\\
& \le \eta_J \| \mf(\vx_c) \| + 2 u_s \| \mf'(\vx^*) \| \| \vs \|.
\end{array}
\endeq

Combining \eqnok{stepest} and \eqnok{inexactok} completes the proof.

\end{proof}

Now suppose we can obtain $\eta_J = O(u_s) = O(\sqrt{u_d})$, then 
\eqnok{errest2} holds and the local improvement estimate becomes
\begeq
\label{eq:errest3}
\| \ve_+ \| = O(\| \ve_c \|^2 + u_d )
\endeq
and the iteration statistics should be the same as Newton's method.
We will see exactly this in the results in \S~\ref{sec:results}.

The assumption that $u_s P^* < 1/2$ simply says that $\mf'(\vx^*)$ is
not horribly ill-conditioned. Ill-conditioning of $\mf'(\vx^*)$ does
not appear in the local improvement estimate directly, but does affect
the convergence of the linear iteration, as we will see in the next section.

\subsection{Iterative refinement}
\label{sec:IR}

Our first choice for an iterative method will be classic iterative
refinement \cite{wilkinson63} for solving a linear system
$\ma \vu = \vb$. Consistently with the application in this paper,
we will assume that the linear system is in single precision  and that
we factor the matrix in half precision. The reader should be aware that
one must store a half precision copy of $\ma$.
The basic algorithm is

\begin{algorithm}
\label{alg:ir}
{$\mbox{\bf IR}(\ma, \vb, \vu)$}
\begin{algorithmic}
\STATE $\vr = \vb - \ma \vu$
\STATE Store $\ma_h = I_s^h(\ma)$
\STATE Factor $\ma_h$ in half precision to
obtain computed factors ${\hat \ml}$ and ${\hat \mU}$.
\WHILE{$\| \vr \|$ too large}
\STATE $\vd = {\hat \mU}^{-1} {\hat \ml}^{-1} \vr$
\STATE $\vu \leftarrow \vu + \vd$
\STATE $\vr = \vb - \ma \vu$
\ENDWHILE
\end{algorithmic}
\end{algorithm}

In Algorithm~\ref{alg:ir} $\vu$ is the initial iterate on input
and the converged solution on output. Note that we are careful to use notation
to stress that we use the computed $LU$ factors in half precision.

One can express the iteration in closed form as
\[
\vu \leftarrow (\mi - {\hat \mU}^{-1} {\hat \ml}^{-1} \ma )\vu
+ {\hat \mU}^{-1} {\hat \ml}^{-1} \vb.
\]
Hence, Algorithm~\ref{alg:ir} is a linear stationary iterative method 
with iteration matrix
\[
\mm  = (\mi - {\hat \mU}^{-1} {\hat \ml}^{-1} \ma ).
\]
So, if the half precision factorization is a sufficiently
good approximation to $\ma$, then
$\| \mm \|$ will be small and iteration will converge rapidly, at least in
exact arithmetic. In the presence of rounding errors we would only expect a
local improvement result. 

Half precision can be very inaccurate and one must be prepared for the iteration
to converge slowly or even diverge. One can show that if the low precision
factorization is a reasonably good approximation to $\ma$, then one obtains
exactly the local improvement results one would like.
One such estimate is from \cite{CarsonHigham} using the $\ell^\infty$
norm on $R^N$. 
In the case here, where $u_h^2 = u_s$, one can show convergence if
\begeq
\label{eq:carson}
3 N u_h \mbox{cond}(\ma) < 1
\endeq
where $\mbox{cond}(\ma) = \| \, | \ma^{-1} | \, |\ma| \, \|_\infty$ and
$| \ma |$ is the matrix with entries $|a_{ij}|$. In that case the iteration
will reduce the linear residual by a factor $O(u_h)$ until
\begeq
\label{eq:backok}
\| \vb - \ma \vu \| = O( u_s \| \vb \| + \| \ma \|_\infty \| \vu \|_\infty).
\endeq

Now we interpret \eqnok{backok} in terms of the inexact Newton conditions
\eqnok{inexact} and \eqnok{inexactJ}. We have $\ma = \mf'(\vx_c)$,
$\vb = \mf(\vx_c)$ and the solution $\vu$ is the Newton step $\vs$. Since
$\| \vs \| = O(\| \mf(\vx_c) \|)$, the estimate \eqnok{backok} implies
\eqnok{inexactJ} with $\eta_J = O(u_s)$.

If the matrix $\ma$ is poorly conditioned, then \eqnok{carson} can fail
and then iterative refinement may fail to converge or fail to satisfy 
\eqnok{backok}. We will see this for the ill-conditioned example
in \S~\ref{sec:results}.

Even if $\| \mm \| > 1$, the condition number of
\[
{\hat \mU}^{-1} {\hat \ml}^{-1} \mj
\]
may be small enough to motivate using a Krylov method with
the low-precision factorization
as a preconditioner.
We use the GMRES-IR approach from \cite{CarsonHigham1,CarsonHigham}
with left preconditioning. This means that we solve 
\[
{\hat \mU}^{-1} {\hat \ml}^{-1} \ma \vd = 
{\hat \mU}^{-1} {\hat \ml}^{-1} \vr
\]
with GMRES to compute the defect $\vd$ in Algorithm~\ref{alg:ir}.
The preconditioner-vector product is computed with two triangular solves.
The results in \S~\ref{sec:results} show how this approach can improve
simple IR in one ill-conditioned case.

\subsection{Interprecision Transfers}
\label{sec:ipxf}

Finally, we must discuss some important details of interprecision
transfers and mixed precision operations.

For the two-precision implementation, when we solve
\begeq
\label{eq:fly}
{\hat \mj} \vs = - \mf(\vx_c)
\endeq
for the Newton step, we need to account for the interprecision
transfers. If we do nothing, then the triangular factors are in
precision $u_J$ and $\mf$ is in double precision. 
In that case each operation in the triangular solves will promote the low
precision matrix elements to double within the CPU registers.
This is called ``interprecision transfer on the fly''.

Interprecision transfer on the fly is $O(N^2)$ work
on interprecision transfers, but can be a noticeable cost for medium
to low dimensions even though the factorization cost is $O(N^3)$ work.
A way to avoid this cost is to round $\mf$ to
precision $u_J$ before the solve. One must take care if $\vx_c$ is near
the solution because rounding down, especially in half precision, could
result in an underflow to zero \cite{highamscaling}. The remedy for this
is to scale $\mf$ to a unit vector before rounding and then reverse
the scaling after the linear solve.
With this in mind one solves 
\begeq
\label{eq:nofly}
{\hat \mj} {\hat \vs}  = - I_d^J (\mf(\vx_c)/\| \mf(\vx_c) \|)
\endeq
entirely in the lower precision. This avoids interprecision transfers
during the triangular solves. The one promotes $\hat \vs$ and
reverses the scaling
\begeq
\label{eq:noflystep}
\vs = \| \mf(\vx_c) \| I_J^d {\hat \vs}
\endeq
to obtain a step $\vs$ in precision $u_J$. Then one would update the
solution via
\[
\vx_+ = \vx_c + \vs.
\]
This is exactly what we do in our Julia codes
\cite{ctk:siamfanl,ctk:sirev20}.
The reader should know that the steps $\vs$ computed with
\eqnok{fly} and \eqnok{nofly}-\eqnok{noflystep}
are different, but the performance of
the nonlinear iteration is unlikely to change.

For the linear iterative refinement iteration, the ideas are similar.
Interprecision transfers on the fly are implicit in our discussion
in \S~\ref{sec:IR} where we view iterative refinement as a 
stationary iterative method. Just as in the nonlinear case, one can
mitigate the interprecision transfer cost by replacing the step
\[
\vd = {\hat \mU}^{-1} {\hat \ml}^{-1} \vr
\]
from Algorithm~\ref{alg:ir} with 
\[
\vd = \| \vr \| I_j^s ({\hat \mU}^{-1} {\hat \ml}^{-1} I_s^h (\vr/|\ \vr \|) ).
\]
The iteration is no longer a stationary iterative method. Instead the
iteration is
\begeq
\label{eq:cheapir}
\vu \leftarrow \vu + 
\| \vb - \ma \vu \|
I_j^s \left({\hat \mU}^{-1} {\hat \ml}^{-1} I_s^h 
\left(\frac{\vb - \ma \vu}{\| \vb - \ma \vu \|} \right) \right).
\endeq
The fixed point map
is nonlinear and, because of the interprecision transfers, not even
continuous. However two approaches to interprecision transfer give 
the same results for all but the most ill-conditioned problems.

For GMRES-IR, however, using \eqnok{cheapir} will not suffice.
One must do the triangular solves in the higher
precision, single precision in the case of this paper, and hence assume
the interprecision transfer cost. One way to mitigate this cost is to
map the half precision factorization of $\mj_h$
to single precision before the solve. The 
cost of this is storage (one more copy of $\mj$), but the on-the-fly
interprecision cost is avoided.

\section{Examples}
\label{sec:results}

In this section we compare some of the two precision
results from \cite{ctk:sirev20} with the three precision method
from this paper. Some of the results using half precision were poor because
the half precision Jacobian was a poor approximation to the Jacobian. This
problem was particularly severe for the ill-conditioned example, which we
feature in this section.
The example, taken from \cite{ctk:sirev20} is the composite mid-point
rule discretization of
the Chandrasekhar H-equation \cite{chand},
\begeq
\label{eq:heq}
\calf(H)(\mu) = H(\mu) -
\left(
1 - \frac{c}{2} \int_0^1 \frac{\mu H(\mu)}{\mu + \nu} \dnu
\right)^{-1} = 0.
\endeq
The nonlinear operator $\calf$ is defined on $C[0,1]$, the space of
continuous functions on $[0,1]$. 

We use $N$ quadrature points $\nu_j = (j - 1/2)/N$ for $1 \le j \le N$
and the rule is
\[
\int_0^1 f(\nu) \dnu \approx \frac{1}{N} \sum_{j=1}^N f(\nu_j).
\]
The discrete system is
\begeq
\label{eq:hmid}
\mf(\vx)_i \equiv
x_i - \left(
1  - \frac{c}{2N} \sum_{j=1}^N \frac{x_j \mu_i}{\mu_j + \mu_i}
\right)^{-1}
=0.
\endeq
As we explained in \cite{ctk:sirev20,ctk:fajulia,ctk:acta}, one can evaluate the
nonlinear residual with a fast Fourier transform to
in $O(N \log(N))$ work and compute an analytic Jacobian, as we did
for this paper, in $O(N^2)$ work.
Hence the dominant cost for large $N$ is the factorization
of the Jacobian.

For $c=.99$, the results in
\cite{ctk:sirev20} showed a significant difference in performance for
the two-precision algorithm between a low precision of single and one
of half (see Figure~3.2, pg 205 and Figure~3.5 pg 208 in \cite{ctk:sirev20}). 
We will reproduce some of those data in this section to make the comparison.

\subsection{Computations}
\label{subsec:compute}
The computations in this section were done in Julia 
\cite{Juliasirev}
v 1.9 on a 2023
Apple Mac Mini with an M2 Pro processor and 32GB of memory.
The M2 processor supports half precision
computing in hardware and Julia 1.9 offers support for this hardware.
However, as we said in the introduction, LAPACK and the BLAS do
not take full advantage of the half precision hardware, 
so half precision computations are slow, 
but not as slow as the ones the author did for \cite{ctk:sirev20}. 

Our implementation of iterative refinement terminates with success
when the relative
residual norm is $< 10^{-6} \approx 100 u_s$ and declares that the
iteration has failed if the residual norm increases. After failure
the algorithm returns the solution of the linear problem
with the best residual.
This approach allows the nonlinear iteration to continue. We see in
the results that failure of the linear iterative refinement iteration
can affect the convergence of the nonlinear solver. The reason for this
is that the inexact Newton condition can fail in this case and convergence
can be slower that expected.

The computations used the author's SIAMFANLEquation.jl
Julia package \cite{ctk:siamfanl, ctk:notebooknl, ctk:fajulia}.
The files for the package are located at
\url{https://github.com/ctkelley/SIAMFANLEquations.jl} and the associated
IJulia notebook can be found at
\url{https://github.com/ctkelley/NotebookSIAMFANL}.
The GitHub repository
\url{https://github.com/ctkelley/Newton3P}
contains the codes used to produce the results in this section and
instructions for reproducing those results.

Table~\ref{tab:99} presents the residual history for ten iterations from
the the ill-conditioned example from \cite{ctk:sirev20}.
The first three histories are for Newton's method with double precision
(F64), single precision (F32), and half precision (F16) Jacobians and
are the same results as those from \cite{ctk:sirev20}. The final two
columns are for IR with the Jacobian stored in single precision
and the factorization done in half precision (IR 32-16) and GMRES-IR
using the that factorization to precondition GMRES.

The Newton iteration with a half precision Jacobian converged very poorly 
for this problem. However, IR with the same half precision factorization
performs as well as Newton's method using a double precision Jacobian. 

\begin{table}[h!]
\caption{Residual Histories for Three Precisions: $N=4096$, $c=.99$}
\label{tab:99}
\begin{center}
\begin{tabular}{llllll} 
 n&        F64&        F32&        F16&   IR 32-16&      IR-GM \\ \hline 
 0 &  1.000e+00 &  1.000e+00 &  1.000e+00 &  1.000e+00 &  1.000e+00  \\ 
 1 &  2.289e-01 &  2.289e-01 &  5.065e-01 &  2.289e-01 &  2.289e-01  \\ 
 2 &  3.934e-02 &  3.934e-02 &  2.958e-01 &  3.934e-02 &  3.934e-02  \\ 
 3 &  2.737e-03 &  2.737e-03 &  1.890e-01 &  2.737e-03 &  2.737e-03  \\ 
 4 &  1.767e-05 &  1.767e-05 &  1.255e-01 &  1.767e-05 &  1.767e-05  \\ 
 5 &  7.486e-10 &  7.536e-10 &  8.518e-02 &  7.538e-10 &  7.506e-10  \\ 
 6 &            &            &  6.068e-02 &            &             \\ 
 7 &            &            &  4.240e-02 &            &             \\ 
 8 &            &            &  3.195e-02 &            &             \\ 
 9 &            &            &  2.280e-02 &            &             \\ 
10 &            &            &  1.713e-02 &            &             \\ 
\hline 
\end{tabular} 
\end{center} 
\end{table}

Table~\ref{tab:9999} is a much more poorly conditioned problem. In this
problem IR does very well for the first three iterations when the 
iteration is not close to the solution. As the iteration converges the
ill-conditioning of the Jacobian at the solution begins to cause failures
of the linear iteration and that affects the nonlinear iteration. GMRES-IR
continues to perform well.

\begin{table}[h!] 
\caption{Residual Histories for Three Precisions: $N=4096$, $c=.9999$}
\label{tab:9999}
\begin{center}
\begin{tabular}{llllll} 
 n&        F64&        F32&        F16&   IR 32-16&      IR-GM \\ \hline 
 0 &  1.000e+00 &  1.000e+00 &  1.000e+00 &  1.000e+00 &  1.000e+00  \\ 
 1 &  2.494e-01 &  2.494e-01 &  5.182e-01 &  2.494e-01 &  2.494e-01  \\ 
 2 &  6.093e-02 &  6.093e-02 &  3.123e-01 &  6.093e-02 &  6.093e-02  \\ 
 3 &  1.480e-02 &  1.480e-02 &  2.067e-01 &  1.480e-02 &  1.480e-02  \\ 
 4 &  3.454e-03 &  3.454e-03 &  1.421e-01 &  3.455e-03 &  3.454e-03  \\ 
 5 &  6.762e-04 &  6.762e-04 &  1.012e-01 &  6.766e-04 &  6.762e-04  \\ 
 6 &  7.049e-05 &  7.049e-05 &  7.552e-02 &  6.360e-04 &  7.049e-05  \\ 
 7 &  1.223e-06 &  1.223e-06 &  5.773e-02 &  5.811e-04 &  1.223e-06  \\ 
 8 &  3.947e-10 &  3.957e-10 &  4.543e-02 &  5.312e-04 &  3.952e-10  \\ 
 9 &            &            &  3.639e-02 &  4.862e-04 &             \\ 
10 &            &            &  2.949e-02 &  4.456e-04 &             \\ 
\hline 
\end{tabular}
\end{center}
\end{table}

\clearpage

\section{Acknowledgments}
This work was partially supported by
%
US Department of Energy grant DE-NA003967
and National Science Foundation Grant
DMS-1906446.

\bibliographystyle{siamplain}
\bibliography{newton3p}


\end{document}